\documentclass[a4paper,11pt]{article}
\usepackage[latin1]{inputenc}
\usepackage{amsfonts}
\usepackage{amsmath}
\usepackage{amssymb}
\usepackage{amsthm}
\usepackage[T1]{fontenc}
\usepackage[dvips]{graphicx}
\usepackage{color}
\usepackage{epsfig}

\theoremstyle{plain}
\newtheorem{Thm}{Theorem}
\newtheorem{Lem}{Lemma}

\newtheorem{Cor}{Corollary}
\newtheorem{Exa}{Example}
\newtheorem{Rem}{Remark}
\newtheorem*{Rem*}{\textsc{Remark}}
\newtheorem*{Def*}{Definition}
\newtheorem*{Prob*}{Direct Monodromy Problem}
\newtheorem*{Lem*}{\textsc{Lemma}}
\newtheorem*{Cor*}{\textsc{Corollary}}
\newtheorem*{Con*}{\textsc{Conjecture}}
\newcommand{\e}{\varepsilon}
\newcommand{\La}{{\cal{L}}}
\newcommand{\eps}{\varepsilon}
\newcommand{\bb}[1]{\mathbb{ #1 }}
\newcommand{\norm}[1]{\left| #1 \right|}
\newcommand{\av}[1]{\left\| #1 \right\|}
\newcommand{\avs}[2]{\left\| #1 \right\|_{H^{#2}}}
\newcommand{\beq}{\begin{equation}}
\newcommand{\eeq}{\end{equation}}

\long\def\symbolfootnote[#1]#2{\begingroup%
\def\thefootnote{\fnsymbol{footnote}}\footnote[#1]{#2}\endgroup}

\title{Semiclassical limit for generalized KdV equations before the gradient catastrophe}
\author{D.~Masoero\thanks{Grupo de F\'isica Matem\'atica,
Complexo Interdisciplinar da Universidade de Lisboa, Av. Prof. Gama Pinto, 2
PT-1649-003 Lisboa,  Portugal.  dmasoero@gmail.com }\,\,\, and A.~Raimondo\thanks{SISSA, Via Bonomea 265, 34136 Trieste, Italy. andrea.raimondo@sissa.it}}
\date{}
\begin{document}

\maketitle

\abstract{We study the semiclassical limit of the (generalised) KdV equation, for initial data with Sobolev regularity, before the time of the gradient catastrophe of the limit conservation law. In particular, we show that in the semiclassical limit the solution of the KdV equation: i) converges in $H^s$ to the solution of the Hopf equation, provided the initial data belongs to $H^s$, ii) admits an asymptotic expansion in powers of the semiclassical parameter, if the initial data belongs to the Schwartz class. The result is also generalized to KdV equations with higher order linearities.}

\section*{Introduction}

We consider the  following class of partial differential equations
\begin{equation}\tag{G}\label{intro:genkdv}
\partial_{t}u= a(u)\,\partial_{x}u+ \sum_{i=1}^n\e_i\,\partial_x^{2i+1}u,\qquad u(x,0)=\varphi(x),
\end{equation}
depending on a family of parameters $\e=\left(\e_1,\dots,\e_n\right)\in\mathbb{R}^{n}$.
Here $a(u)$ is a smooth function, $x\in\bb{R}$, and the initial data is independent of $\e$.
We call the above family of equations \emph{generalized KdV equations}, for it contains as a particular
example the KdV equation itself:
\begin{equation*}
\partial_{t}u= u\,\partial_{x}u+ \e\,\partial_x^{3}u \;.
\end{equation*}
Other examples are given by the Kawahara equation \cite{kawa72}, which is obtained by choosing
$a(u)=u$ and $n=2$, as well as nonlinear generalizations of KdV, for $n=1$ and $a$ arbitrary.\\

We are interested in the behaviour of the solutions of \eqref{intro:genkdv}
as the parameters $\e$ vary. In particular, we consider the behaviour when $\e\to 0$, in which limit, formally, \eqref{intro:genkdv}
becomes a quasilinear conservation law, of the form:
\begin{equation}\tag{H}\label{intro:hopf}
\partial_{t}u= a(u)\,\partial_{x}u,\qquad u(x,0)=\varphi(x).
\end{equation}
The case $a(u)=u$ is known as Hopf equation. The problem of studying solutions of equation \eqref{intro:genkdv} as $\e \to 0$ is known in the literature as the semiclassical (or singular, or dispersionless) limit.
Since solutions of equation \eqref{intro:hopf} may develop a singularity at a critical time
$0^+\leq t_c <\infty$, we study local-in-time solutions of \eqref{intro:genkdv} for those classes of initial data, to be specified below,
for which the Cauchy problem for the generalized KdV equation is locally well-posed uniformly with respect to $\e \in \bb{R}^n$.\\

There is an extensive literature on the initial value problem for generalized KdV equations, which is based essentially on
two distinct approaches: the first makes use of the inverse scattering and Riemann-Hilbert methods, and it applies to those
equations of the class \eqref{intro:genkdv} which are integrable, such as KdV. The second applies to a more general class of equations, and makes use of
fix-point arguments for the associated integral equation
\begin{align}\tag{W}
 & u(t;\e)=W(t;\e)u(0;\e)+\int_0^t W(t-s;\e) a(u(s;\e))u_x(s;\e) ds \, ,\\ \nonumber
& W(t;\e)= \exp{\big( t \sum_{i=1}^n \e_i \partial_x^{2i+1}  \big)}  \; .
\end{align}

Since the seminal paper of Lax and Levermore   \cite{lale83}, the method of inverse scattering has been successfully used  to study the semiclassical limit of the KdV equation, both before and after the critical time $t_c$ of the Hopf equation.
For time smaller than $t_c$, rigorous results have been obtained for those initial data whose scattering transform can be computed
in the semiclassical limit using the WKB analysis. In this case, the corresponding solutions are proven to converge in $L^2$
to the solutions of the Hopf equation \cite{lale83,venakides85}.
If, in addition,  the initial data satisfy some analyticity assumptions, then the powerful nonlinear
steepest-descent analysis \cite{deift93} applies to the study of the semiclassical limit. For these initial data,
the solutions are known to converge uniformly - see \cite{deift98,claeys2009}. \\ 

The integral equation \thetag{W} has been a major topic of investigation since it was used by Kenig, Ponce and Vega \cite{kenig91} to
establish local well-posedness of \eqref{intro:genkdv} for polynomial nonlinearity $a(u)$ and dispersion $\e \neq 0$.
In the particular case of the KdV equation, the authors of \cite{kenig91} have been able to prove local well-posedness in $H^s, s>\frac{3}{4}$.
Their results were then refined by many authors to obtain local and
global well-posedness for low-regular initial data, see for instance
\cite{bourgain93, kenig96,colliander2003}.

The theory of equation \thetag{W} relies heavily on the dispersive character of
equation \eqref{intro:genkdv}, being based on
the smoothing effects of the linear evolution operator $W$. Hence, it may
be very difficult to apply it to the study of the semiclassical limit.
Moreover, since the equation \thetag{H} may be ill-posed for $s\leq\frac{3}{2}$,
it seems unreasonable to study the semiclassical limit for low-regular solutions.\\

Due to the limitations of the inverse scattering transform and of the
integral equation \thetag{W},
in order to deal with the semiclassical limit of the general equation \eqref{intro:genkdv} we choose a different approach, namely Kato's theory of
quasi-linear equations \cite{ka75}. This approach turns out to be particularly suitable to our problem because it allows to treat equation \eqref{intro:genkdv}
for all values of $\e \in \bb{R}^n$ on the same footing, and it is
very robust under perturbations.\footnote{An alternative method, which may be suitable for the study of the semiclassical limit, is given by the Bona-Smith energy method \cite{bosm75}.}\\

Following Kato, we consider \eqref{intro:genkdv} as a quasilinear equations on a
Banach space $X$, of the form
\begin{equation}\tag{Q}
\frac{du(\e)}{dt}= A(u(\e); \e)u + f(u(\e);\e),\quad  0\leq t\leq
T,\quad u(0)=\varphi.
\end{equation}
Here  $A(y; \e)$ is a linear operator,
depending on  $\e \in \bb{R}^n$ and on some element $y\in X$. In addition, for
any fixed $\e$ and $y$
the operator $A(y;\e)$ generates a $C_0-$semigroup on $X$. In our case,
$$A(y;\e)=a(y)\,\partial_{x}+ \sum_{i=1}^n\e_i\,\partial_x^{2i+1},\qquad f=0,$$
and we choose $X=L^2(\bb{R})$ . Kato himself used his theory to construct local--in--time  solutions of equation
\eqref{intro:genkdv}. In particular,
he established local well-posedness for the KdV equation in $H^s, s >
\frac{3}{2}$, both when $\e \neq 0$ \cite{ka83}
and when $\e=0$ \cite{ka75}.
However, he did not consider the semiclassical limit.

We establish  simple conditions, under which the local-in-time solution of the Cauchy problem \eqref{intro:genkdv} with initial datum in $H^{s}$ is continuous
and $N-$differentiable with respect to $\e$.

Essentially, we show that:
\begin{itemize}
\item If $s\geq 2n+1$, the local-in-time solution $u(\e)$ of
\thetag{G} is continuous with respect to $\e \in \bb{R}^n$ \; .
\item Let $K=\sum_{i=1}^{n}N_i(2i+1)$ and $N=\sum_{i=1}^{n}N_i$. If
$s-K \geq 2n+1$, then the partial derivative
$$\frac{\partial^N u(\e)}{\partial \e_1^{N_1}\dots\partial \e_n^{N_n}}$$ 
exists in  $H^{s-K}$ and it is continuous with respect to $\e \in \bb{R}^n$.
\end{itemize}

\vspace{20pt}

The above results can be applied to the project, proposed by Du\-bro\-vin and Zhang \cite{duzh01,du06}, of Hamiltonian perturbations of
quasilinear conservation laws \thetag{H}.
Indeed, any equation the form \thetag{H} can be written as an infinite dimensional Hamiltonian system;  within this theory,
one looks for suitable deformations -- depending on arbitrary functions of $u$ and its derivatives -- such that the equation
remains Hamiltonian. One simple example is given by KdV, which can be obtained as a Hamiltonian perturbation of the Hopf equation. 

In addition to the above problem, the project includes a classification of integrable perturbations and a characterization of
the solutions of the perturbed equations, both before the critical time and in a neighborhood of it.
It should also be noted that the project applies not only to single equations, but also to systems of first oder quasilinear PDEs, \cite{du06,du10,duzh01} . \\

Before critical time, the Dubrovin-Zhang theory provides a way to construct solutions of the perturbed equation
in terms of solutions of the unperturbed one. Without going into detail, one looks for solutions of the perturbed equations
as formal power series $u=v^{0}+v^{1}\,\e^{2}+\dots,$ and argues that $v^{0}$ is a solution of the unperturbed equation,
while the subsequent coefficients can by obtained from $v^{0}$ by a recursive procedure. This construction, although extremely
powerful in predicting behaviour of the solutions, is based on formal identities, and requires rigorous justification.\\

The paper is organized as follows: after recalling in Section \ref{sec:semigroups} some basic elements of the theory of $C_{0}-$semigroups,
we consider in Section \ref{sec:kato} the results obtained by Kato in \cite{ka75}, which we will use in order to prove our results.\\

Sections \ref{sec:continuity} and \ref{sec:kdvderivatives} are the core of the paper. We first prove the existence of a positive time $T$
for which the solutions of the problem \eqref{intro:genkdv}, with initial data in $H^{s}$, $s\geq 2\,n+1$,  are continuous functions of
the parameters $\e$. In particular, such solutions are continuous as $\e\to 0$, implying $H^{s}-$convergence to the solution of \thetag{H}
in this limit. This result generalizes the one obtained by Lax and Levermore (before the critical time) to equations of
type \eqref{intro:genkdv} -- which are not necessarily integrable -- and to initial data in the Sobolev space $H^{s}$.\\

In Section \ref{sec:kdvderivatives} we consider differentiability of solutions of \eqref{intro:genkdv} with respect to $\e$.
Although our result   -- with suitable modifications -- holds for every equation of type \eqref{intro:genkdv}, for simplicity we
consider in detail the KdV example only. We show that if the initial datum of the KdV equation lies in the Sobolev space $H^{s}$,
then the solution of the Cauchy problem is $N-$times differentiable with respect to $\e$, for $N=\lfloor s/3-1\rfloor$.\\

In Section \ref{sec:dubrovin} we present the Dubrovin-Zhang theory of Hamiltonian perturbations of equation \thetag{H}, considering those
aspects of the theory which are directly related to the results of the present paper: the classification results of Hamiltonian
perturbations and the construction of the solutions of the perturbed equations before critical time.\\

In the last Section we apply the results obtained in Sections \ref{sec:continuity}  and \ref{sec:kdvderivatives} to the Dubrovin's theory.
Since equations of type \eqref{intro:genkdv} can be seen as Hamiltonian perturbations of \thetag{H}, we provide -- for this class of equations --
a rigorous justification to the heuristic results of Section \ref{sec:dubrovin}. In addition, we find an explicit formula for the coefficient $v^{1}$
of the solution of a generic Hamiltonian perturbation of \thetag{H} in terms of the solution $v^{0}$ of the unperturbed equation. This is of the form:
\begin{equation*}
v^{1}=\frac{t}{2}\,\frac{\partial}{\partial x}\! \left(\frac{\left(c\,a'\right)'(v^{0}_{x})^{2}+2\,c\,a'\,v^{0}_{xx}+t\,c\,(a')^{2}v^{0}_{x}\,v^{0}_{xx}+t\,c'\,(a')^{2}(v^{0}_{x})^{3}}{(1+t\,a'\,v^{0}_{x})^{2}}\right),
\end{equation*}
where $a=a(v^{0})$ is the non-linearity of \thetag{H}, and $c=c(v^{0})$ is a function characterizing the Hamiltonian perturbation.

\subsection*{Notation}
Given the real Banach spaces $X,Y,\dots$, we let $\av{\;\,}_X, \av{\;\,}_Y, \dots $ denote the corresponding norms.
$\La(Y,X)$ denotes the Banach space of bounded linear operators from $Y$ to $X$, with norm $\av{\;\,}_{Y,X}$, while $\La(X)$ denotes the Banach space of bounded linear operator from $X$ to itself with the norm $\av{\;\,}_X$.
We call $D(A)$ the domain of an operator $A$. $L^2:=L^2(\bb{R})$ denotes the Hilbert space of square integrable real functions and
$H^s:=H^s(\bb{R}), s \geq 0$ denotes the Sobolev space of order $s$.
The symbol $\partial^n_x$ denotes the $n$-th derivative with respect to $x$ or the corresponding operator on $L^2$
with domain $H^n$.\\

In the present paper we consider real Banach spaces and real functions only.

\paragraph{Acknowledgments}
We are grateful to Boris Dubrovin, Percy Deift, Ken McLaughlin and Tamara Grava
for encouraging us in our research. A.R. and D.M. thank, respectively, the Grupo de F\'isica Matem\'atica da Universidade de Lisboa and the SISSA Mathematical
Physics sector for the kind hospitality.

The research was partially supported by the INDAM--GNFM \lq Progetto Giovani $2010$\rq. D.M. is supported by a Postdoc scholarship of the Funda\c{c}\~ao para a Ci\^encia e a Tecnologia,
project PTDC/MAT/104173\-/2008 (Probabilistic approach to finite and infinite dimensional dynamical systems).

\section{$C_0-$semigroups}\label{sec:semigroups}
The theory of $C_0-$semigroup is a standard tool in analysis.
Following \cite{ka75} and \cite{pazy83}, we recall the elements of the theory we will use in the rest of the paper.

A one parameter family of linear operators $\left\{T(t), \,0\leq t <\infty\right\}$ on a Banach space $X$ is a \emph{$C_0-$semigroup} if it
is a strongly continuous semigroup of bounded linear operators, namely, it satisfies:
\begin{itemize}
 \item $T(0)=I,\qquad T(t)T(s)=T(t+s),\qquad t,s\geq 0$,
\item $\lim_{t \downarrow 0} T(t)x=x,\qquad \forall x \in X$.
\end{itemize}
Here $I$ is the identity operator on $X$. The linear operator defined by
\begin{equation*}
 D(A)=\left\lbrace x \in X:\lim_{t \downarrow 0} \frac{T(t)x-x}{t} \;\; \mbox{exists } \right\rbrace
\end{equation*}
and
\begin{equation*}
A x=\lim_{t \downarrow 0} \frac{T(t)x-x}{t}, \qquad \forall x \in D(A),
\end{equation*}
is the \textit{infinitesimal generator} of the $C_{0}-$semigroup $\left\{T(t)\right\}$. The operator $A$ is closed and densely defined.

A standard theorem shows that for any $C_0-$semigroup there exist
two positive constants $M\geq1$ and $\beta\geq0$ such that
\begin{equation}\label{eq:constants}
 \av{T(t)}_X\leq Me^{\beta t}, \qquad 0\leq t <\infty .
\end{equation}
In particular, a $C_0-$semigroup with constants $M=1,\beta=0$ is called a \emph{semigroup of contractions}. We denote by $G(X,M,\beta)$ the set of infinitesimal generators of $C_0-$semigroups with constants $M$, $\beta$. 
\begin{Rem}
Note that if $A \in G(X,1,\beta)$, then $A-\beta I \in G(X,1,0)$.
\end{Rem}
Later on we will need a perturbation theorem for generators of semigroups of contractions; for this purpose we introduce the following notions.

An operator $A$ on a Hilbert space $X$ is said to be \emph{dissipative} if for every $x \in D(A)$, we have $Re(Ax,x)\leq0$. 

If $A$ and $B$ are operators on a Banach space $X$, we say that $B$ is \emph{relatively bounded} with respect to $A$ with relative bound
$\rho \geq 0$ if $D(A) \subset D(B)$ and there exists a $\sigma \geq0$ such that
\begin{equation}\label{eq:relativebound}
\av{Bx}_X\leq \rho \av{Ax}_X +\sigma\av{x}_X, \quad\forall x \in D(A).
\end{equation}

\begin{Thm}\label{thm:perturbation}
 Let $A \in G(X,1,0)$ be the generator of a $C_0-$semigroup of contractions on a Hilbert space $X$. Let $B$ be dissipative and relatively bounded
with respect to $A$ with relative bound $\rho<1$.
Then $A+B$ is the generator of a semigroup of contractions.
\begin{proof}
See \cite{pazy83}, Corollary 3.3
\end{proof}

\end{Thm}
The definition of dissipative operator and Theorem \ref{thm:perturbation} can be generalised to any Banach space,
with slight modifications. However, such generalisation is not necessary in our study. The following examples will be useful in the rest of the paper.

\begin{Exa}\label{ex:d3}
Let $X$ be a Banach space and $A$ be an anti-self-adjoint operator.
Due to the Stone Theorem, $A$ generates a $C_0-$group of unitary operators,
hence a $C_0-$semigroup of contractions. For instance, the derivative operator $\partial^{2n+1}_x$, $n \in \mathbb{N}$, with domain $D(\partial^{2n+1}_x)=H^{2n+1}$, is an anti-self-adjoint operator
on the space $L^2$. More generally, the operator 
\begin{equation}
D^{2n+1}_{\e}=\sum_{i=0}^{n}\e_i\partial^{2i+1}_x,\qquad \e=(\e^{1},\dots,\e^{n}) \in \bb{R}^{n},
\end{equation}
with domain $H^{2n+1}$ is anti-self-adjoint on $L^2$ for any value of the parameter $\e$.
\end{Exa}

\begin{Exa}\label{ex:fd}\cite{ka75}
Let $f(x)$ be a bounded differentiable function with bounded derivative on the whole real axis, take $X=L^2(\mathbb{R})$, $B=f(x) \partial_x$ and $D(B)=H^1$.
The operator $B$ can be decomposed into an anti-self-adjoint part and a bounded self-adjoint part, $B=B_{1}+B_{2}$ with:
$$B_{1}= (f\partial_x+ \frac{1}{2}f_x),\qquad B_{2}=-\frac{1}{2}f_x \;.$$
In particular, $B_{1}$ has domain
$H^1$ and is anti-self-adjoint, thus, it generates a $C_0-$semigroup of contractions. On the other hand, $B_{2}$ is a bounded self-adjoint operator with norm  $\av{B_{2}}_{L^{2}}=\frac{1}{2}\sup_{x \in \mathbb{R}}|f_{x}(x)|$, and a simple computation shows that $ B_{2} -\av{B_{2}}_{L^{2}} I $ is dissipative. Due to Theorem \ref{thm:perturbation}, we have that $B-\av{B_{2}}_{L^{2}} I$
generates a $C_0-$semigroup of contractions. Therefore, $B \in G(X,1,\beta)$ with $\beta=\av{B_{2}}_{L^{2}}$. In addition, $B$ is relatively bounded with respect to
$\partial_{x}^{2n+1}$, $n\geq 1$ with any relative bound $\rho>0$. Due to Theorem \ref{thm:perturbation} and the above discussion, 
$$B+\e\, \partial_{x}^{2n+1} \in G(X,1,\beta),$$
for any $\e \in \mathbb{R}$. 
\end{Exa}

\begin{Exa}\label{ex:multiplication}
Let $g(x)$ be a continuous bounded function on $\bb{R}$, take $X=L^2(\mathbb{R})$, and let $C=g(x)$ be the operator of multiplication by $g$, with $D(C)=L^2$.
We have that $C$ is a bounded operator, with norm $\beta'=sup_{x\in \bb{R}}\norm{g(x)}$.
Due to Theorem \ref{thm:perturbation}, if $A \in G(X,1,\beta)$ then
$A+C \in G(X,1,\beta+\beta')$. 
\end{Exa}

\section{Kato's theory on quasilinear equations}\label{sec:kato}
In this section we review Kato's results  on quasilinear equations on a Banach space $X$, of the form 
\begin{equation}\label{eq:general}
\frac{d u}{dt}=A(t,u)\,u+f(t,u),\quad 0\leq t\leq T, \quad u(0)=\varphi.
\end{equation}
Here $A(t; y)$ is a linear operator, depending on the time $t$ and on
some element $y \in X$, and such that for any fixed  $t$ and $y$ the operator $A(t; y)$ generates a $C_{0}-$semigroup on $X$. It should be noted that we do not present these theorems in their strongest form, but in a form adequate to our purpose. For the reader's convenience, we follow -- as much as possible -- the notation of the original paper \cite{ka75}.

First, we consider the linear case, when the operator $A$ and the forcing term $f$ do not depend on $u$. 
\begin{Thm}\label{thm:linear}
The linear non-homogeneous Cauchy problem 
\begin{equation}\label{eq:linear}
 \frac{du}{dt}= A(t)u + f(t), \quad 0\leq t\leq T, \quad u(0) \in Y,
\end{equation}
has a unique solution 
$$u(t) \in C([0,T];Y) \cap C^1([0,T];X),$$ 
provided the following assumptions are satisfied:
\begin{itemize}
     \item[(i)] $X$ is a Banach space and $Y\subset X$ is another Banach space, continuously and densely embedded in $X$. Moreover, there exists an
isomorphism $S$ of $Y$ into $X$.
 \item[(ii)]There exists a positive $\beta$ such that $A(t) \in G(X,1,\beta)$ $\forall t \in [0,T]$.
\item[(iii)]$SA(t)S^{-1}-A(t)=B(t) \in \La(X)$, and $t \mapsto B(t)$ is a continuous operator valued function.
\item[(iv)]$Y \subset D(A(t))$ so that $A(t) \in \La(Y,X)$, and $t \mapsto A(t)$ is a continuous operator valued function.
\item[(v)]$u(0) \in Y$ and $ f\in C([0,T];Y)$.
\end{itemize}
\begin{proof}
See \cite{ka73}, Theorem I and II. 
\end{proof}
\end{Thm}

The following is a perturbation theorem for the linear equation \eqref{eq:linear}.
\begin{Thm}\label{thm:linearpert}
In addition to the assumptions of Theorem \ref{thm:linear}, consider the sequence of Cauchy problems
 \begin{equation}\label{eq:linearpert}
 \frac{du^n}{dt}= A^n(t)u^n + f^n(t),\quad 0\leq t\leq T, \quad u^n(0) \in Y,
\end{equation}
and suppose that, for any fixed $n$, the operator $A^n$ satisfies conditions (i) through (v) of Theorem \ref{thm:linear}.
Moreover, suppose that
\begin{itemize}
 \item[(vi)]$A^n(t) \to A(t)$ strongly in $\La(Y,X)$, and $\sup_{t \in [0,T]}\av{A^n(t)}_{Y,X}$ is uniformly bounded in n.
\item[(vii)]$B^n(t) \to B(t)$ strongly in $\La(X)$, and $\sup_{t \in [0,T]}\av{B^n(t)}_{X}$ is uniformly bounded in n.
\item[(viii)]$u^n(0) \to u(0)$ in $Y$ and $f^n \to f$ in $C([0,T],Y)$.
\end{itemize}
Then  $u^n(t) \to u(t)$ in $C([0,T],Y) \cap C^1([0,T],X)$, where
$u^n(t)$ is the unique solution of (\ref{eq:linearpert}) and $u(t)$ is the unique solution of (\ref{eq:linear}).
\begin{proof}
See \cite{ka73} Theorem V-VI.
\end{proof}
\end{Thm}

We now move to the analogue results for quasilinear equations of the form \eqref{eq:general}. For our purposes, it is sufficient to consider only the homogeneous case, when $f=0$, namely:
\begin{equation}\label{eq:quasilinear}
  \frac{du}{dt}= A(t,u)u, \quad 0\leq t\leq T, \quad u(0) \in W \subset Y \; .
\end{equation}
Here the set $W$ is a bounded subset of $Y$. Due to the nonlinearity, existence of the solutions is not guaranteed on the whole time interval $[0,T]$, but -- in general -- only for a smaller time $T'$, with $0 < T' \leq T$. The reason we restrict to a bounded subset $W$ is because we expect the time of existence to depend on the norm of the initial data. \\

\noindent
Following Kato, we make the following assumptions:
\begin{itemize}
  \item[(X)] $X$ is a reflexive Banach space and $Y\subset X$ is another reflexive Banach space,
continuously and densely embedded in $X$. There is an
isometric isomorphism $S$ of $Y$ into $X$. Moreover, we fix a ball $W\subset Y$ of radius $R$ and centered in $0$.
 \item[(A1)]There exists a positive $\beta$ such that $A(t,y) \in G(X,1,\beta)$, for all $t \in [0,T]$ and $y \in W$.
\item[(A2)]For any $t,y \in [0,T]\times W$, we have 
$$SA(t)S^{-1}-A(t)=B(t) \in \La(X),$$ 
and $\av{B(t)}_X \leq \lambda_1.$
\item[(A3)]For all $t,y \in [0,T]\times W$, we have $A(t) \in \La(Y,X)$. Fixed $y \in W$,  the function $t \mapsto A(t,y)$ is a
continuous operator-valued function, and
fixed $t \in [0,T]$, $y \to A(t,y)$ is Lipschitz continuous, in the sense that there exists a $\mu_1$ such that
$$\av{A(t,y)-A(t,z)}_{Y,X}\leq\mu_1\av{y-z}_X.$$
\end{itemize}

\begin{Thm}\label{thm:quasilinear}
Suppose conditions (X),(A1),(A2),(A3) are satisfied. Then, the quasilinear homogeneous Cauchy problem (\ref{eq:quasilinear})
has a unique solution 
$$u(t) \in C([0,T'];W) \cap C^1([0,T'];X),$$ 
for some $0<T'\leq T.$ In addition, $T'$ has a lower bound uniquely depending on $\beta,\lambda_1,\mu_1, R$ and monotonically decreasing
in each variable. 
\begin{proof}
See \cite{ka75} Theorem 6. 
\end{proof}
\end{Thm}

We now state the perturbation theorem in the case of quasilinear equations of type \eqref{eq:quasilinear}.
\begin{Thm}\label{thm:quasilinearpert}
In addition to the assumptions of Theorem \ref{thm:quasilinear}, consider the sequence of quasilinear homogeneous Cauchy problems
 \begin{equation}\label{eq:quasilinearpert}
 \frac{du^n}{dt}= A^n(t,u^n)u^n,\quad 0\leq t\leq T,\quad u^n(0) \in W,
\end{equation}
and assume that conditions (X), (A1), (A2) and (A3) of Theorem \ref{thm:quasilinear} are satisfied for every $n$,
with constants $\beta,\lambda_1,\mu_1$ independent on $n$. Moreover, suppose that
\begin{itemize}
\item[(A4)]$\av{B^n(t,y)-B^n(t,z)}_X \leq \mu_2\av{y-z}_Y$, uniformly in $n$.
 \item[(C1)]$A^n(t) \to A(t)$ strongly in $\La(Y,X)$.
\item[(C2)]$B^n(t) \to B(t)$ strongly in $\La(X)$.
\item[(C3)]$u^n(0) \in W$ and $u^n(0) \to u(0)$ in $Y$, as $\,\,n\mapsto\infty$.
\end{itemize}
Then, there exists a positive time $0<T''\leq T$, such that there is a unique solution 
$$u^n\in C([0,T''],W) \cap C^1([0,T''],X)$$
of (\ref{eq:quasilinearpert}), for every $n$, and a unique solution $u$ of (\ref{eq:quasilinear}) in the same class.
Moreover 
$$u^n(t) \to u(t)\qquad \text{in}\qquad C([0,T''],W) \cap C^1([0,T''],X),$$
as $\,\,n\mapsto\infty$. The time $T''$ has a lower bound uniquely depending on $\beta,\lambda_1,\mu_1,\mu_2, R,$ and monotonically decreasing
in each variable.
\end{Thm}
\begin{proof}
See \cite{ka75}, Theorem 7. 
\end{proof}

\section{Continuity of solutions for generalised KdV equations}\label{sec:continuity}
Here we apply Kato's theory to study local-in-time solutions of the following family of Cauchy problems:
\begin{eqnarray}\label{eq:genkdv}
   \frac{du}{dt}= A(u;\,\e)u, \qquad u(0) \in H^s , \; \\ \nonumber
A(u;\,\e)=a(u)\partial_x+ \sum_{i=1}^n\e_i\partial_x^{2i+1}  \; ,
\end{eqnarray}
depending on a family of parameters $\e=\left(\e_1,\dots,\e_n\right)\in\mathbb{R}^{n}$.
Here $a(u)$ is a smooth function, the initial value $u(0)$ is independent on the $\e_i$ for all $i$, and
$s\geq 2n+1$. We call the above equations \emph{generalized KdV equations}. We are interested in the behaviour of the solutions of \eqref{eq:genkdv}
as the parameters $\e$ vary; in particular, we want to prove the continuity of the solutions with respect to $\e$ in a suitable Banach space.

To apply the results of the previous section we choose $X=L^{2}$ and $Y=H^{s}$.
Before proving the main result, we collect some known facts about Sobolev spaces (in one space dimension).

\begin{Lem}\label{lem:sobolevestimates}
 \begin{itemize}
\item[(0)]The map $\Lambda^s=(1-\partial_x^2)^{\frac{s}{2}}$ is an isometric isomorphism of $H^s, s \geq 0$ into $L^2$.
The inverse of $\Lambda^s$ is $\Lambda^{-s}$.
  \item[(i)] Sobolev embedding (particular case): if $u \in H^s, s>\frac{1}{2}+n$, then $u$ is
n-times differentiable and
 there exists a constant $c$ such that
 \begin{equation}\label{eq:sobolevemb}
\av{\partial_x^n u}_{L^{\infty}} \leq c\,\avs{u}{s}  \; .
\end{equation}
\item[(ii)]Algebra property: if $u,v \in H^s$, $s>\frac{1}{2}$, then there exists a constant $c(s)$ such that
\begin{equation}\label{eq:algebra}
 \avs{u\,v}{s} \leq c(s) \avs{u}{s} \avs{v}{s}  \; .
\end{equation}
\item[(iii)]Schauder estimate: if $a:\mathbb{R} \to \mathbb{R}$ is a smooth function and $s>\frac{1}{2}$, then there exists a
constant $c(s,a,\avs{u}{s},\avs{v}{s})$ such that
\begin{equation}\label{eq:schauder}
 \avs{a(u)-a(v)}{s} \leq c(s,a,\avs{u}{s},\avs{v}{s})\avs{u-v}{s} \; .
\end{equation}
  \item[(iv)]Commutator estimates: let $u \in H^s, s>\frac{3}{2}$. Then the operator $T_u=(\Lambda^su\partial_x - u\partial_x\Lambda^s)
\Lambda^{-s}$ is bounded on $L^2$ and there exists a constant $c(s)$ such that
\begin{equation}\label{eq:commutator}
 \av{T_u}_{L^2} \leq c(s) \avs{u}{s}  \; .
\end{equation}
\end{itemize}
\end{Lem}

\begin{proof}
For $(iv)$ see \cite{ka75}, Lemma A.2. For all other statements, see \cite{tao2006} Appendix A.
\end{proof}

We prove the following
\begin{Thm}\label{thm:localgenkdv}
Let $W$ be a ball of radius $R$
in $H^s$, such that $u(0) \in W$.
There is a time $T$ such that for any $\e \in \bb{R}^n$ the Cauchy problem (\ref{eq:genkdv}) has a unique solution
$$u(t;\e) \in C([0,T],W)\cap C^1([0,T],H^{s-(2n+1)}).$$ 
The map $\e \mapsto u(t;\e)$ from $\bb{R}^n$ to $C([0,T],W)\cap C^1([0,T],H^{s-(2n+1)}) $ is continuous.

\begin{proof}
As a first step we prove that $u(t;\e)$ exists and is continuous as a map from $\bb{R}^n$
to $ C([0,T],H^s)\cap C^1([0,T],L^2)$.
To prove this, it is sufficient to show that all conditions of Theorems \ref{thm:quasilinear} and \ref{thm:quasilinearpert}
are satisfied uniformly in $\bb{R}^{n}$, assuming $X=L^2(\bb{R})$ and $Y=H^s(\bb{R})$.

\begin{itemize}
 \item[(X)]It is trivially satisfied since $X$, $Y$ are Hilbert spaces. The required isometry can be chosen to be $\Lambda^s$. 
\item[(A1)]Following Example \ref{ex:fd} above, we see that $A(u;\,\e) \in G(X,1,\beta(R))$,
with $$\beta(R)=\frac{1}{2}\sup_{u \in W}\sup_{x \in \bb{R}}\norm{\partial_xa(u)}\leq \frac{1}{2} R \sup_{\norm{x}\leq R}a'(x)\;.$$
\item[(A2)]Since $\Lambda^s$ commutes with the derivative operator, we have 
$$B(u,\e):=(\Lambda^s A(\e) - A(\e)\Lambda^s)\Lambda^{-s}=T_{a(u)},$$ 
where $T$ is defined as in Lemma \ref{lem:sobolevestimates} (iv). Due to the commutator estimate and the Schauder estimate,
we have that
$$\av{B(u;\e)}_{L^2(\bb{R})} \leq c(s,a,R) R = \lambda_1(R) \; .$$
\item[(A3)]$A(u;\e)$ is a continuous operator from $H^s$ to $H^l$, for $0\leq l\leq s-(2n+1)$.
Indeed, $ \sum_{i=1}^n\e_i\partial_x^{2i+1}$ is continuous from $H^s$ to $H^l$ and
$a(u)\partial_x $ is continuous from $H^s$ to $H^{s-1}$ due to the Schauder estimate and the algebra
property of Sobolev spaces. A simple computation shows that $$\av{(A(u;\,\e)-A(v;\,\e))}_{H^{S},L^2}\leq
c(s) \sup_{\norm{x}\leq R}\norm{a'(x)}\av{u-v}_{L^2} ,$$
for some constant $c(s)$ depending on $s$ only.
We choose
$$\mu_1(R)=c(s) \sup_{\norm{x}\leq R}\norm{a'(x)}.$$
\item[(A4)]The same reasoning as in (A2) above, shows that property $A(4)$ is satisfied with $\mu_2(R)=c(s,a,R)$
for some constant $c(s,a,R)$.
\item[(C1)]The operator $A(u;\e)$ depends continuously (in norm) on $\e$. Indeed, if $\e'=(\e'_{1},\dots,\e'_{n})$, we have
$$\av{(A(u;\,\e)-A(u;\,\e'))}_{H^{S},L^2}\leq \sum_{i=1}^{n}\norm{\e_i-\e'_i} .$$
\item[(C2)]The operator $B(u;\,\e)$ does not depend on $\e$.
\item[(C3)]The initial data do not depend on $\e$.
\end{itemize}
We have thus proved that there exists a $T> 0$ such that $u(t;\e)$ exists and is continuous as a map from $\bb{R}^n$
to $ C([0,T],H^s)\cap C^1([0,T],L^2),$ we want to prove that it is continuous also as a map from $\bb{R}^n$ to $ C([0,T],H^s)\cap C^1([0,T],H^{s-(2n+1)}).$
Now the time derivative of $u$ satisfies
$$u_t(t;\e)=A(u(t);\,\e)u(t;\,\e),$$ and
$A(u(t);\e)$, for fixed $t,\e$ is a continuous operator from $H^s$ to $H^{s-(2n+1)}$. To complete the proof, it is enough to prove that
the map  
$$
A(\,,\,): W \times \bb{R}^n \to \La(H^s,H^{s-(2n+1)})
$$
is continuous. Indeed,
\begin{eqnarray*}
\avs{A(v,\e')y-A(u,\e)y}{s-(2n+1)} \!\! & \!\! \leq \!\! & \!\! \avs{a(v)\partial_x y-a(u)\partial_xy}{s-(2n+1)}+ \\
& \!\!& \!\! \avs{\sum_{i=1}^n(\e_i'-\e_i)\partial^{2i+1}_x y}{s-(2n+1)} .     
\end{eqnarray*}
By Cauchy-Schwarz and the Schauder estimate we have 
$$\avs{a(v)\partial_x y-a(u)\partial_xy}{s-(2n+1)} \leq c(s,a,R) \avs{u-v}{s} \avs{y}{s} .$$
Moreover $\avs{\sum_{i=1}^n(\e_i'-\e_i)\partial^{2i+1}_x y}{s-(2n+1)}\leq  \sum_{i=1}^n\norm{\e_i'-\e_i}\avs{y}{s}$.
\end{proof}
\end{Thm}

\begin{Rem}
Following \cite{ka83}, it is possible to prove a slightly stronger version of this theorem. Indeed, we can prove that
the Cauchy problem \eqref{eq:genkdv} is uniformly locally well-posed in $H^s$ with $s>\frac{3}{2}$. However, this stronger
result is unnecessary for the purpose of studying the derivatives (with respect to $\e_i$) of the solution of the Cauchy problem \eqref{eq:genkdv}. 
\end{Rem}

Theorem \ref{thm:localgenkdv} establishes that, fixed the nonlinearity $a(u)$ and the initial datum $\varphi \in H^s$,
there exists a time $T>0$ such that the Cauchy problem (\ref{eq:genkdv}) is locally well-posed in the time interval $[0,T]$,
continuously with respect to $\e \in \bb{R}^n$. In particular the life-span of the solution can be chosen independently on
$\e$.\\

The natural problem is to find the supremum of all the positive times such that the Cauchy problem is
locally well-posed and continuous with respect to $\e \in U$ for some open $U \subset \bb{R}^n$. We denote this time $T_U$. Suppose that for some $a(u)$ and some $\varphi$, there exists a $U\subset \bb{R}^n$, such that
the Cauchy problem is globally well-posed for all $\e \in U$. In this case, it follows that $T_U = \infty$. For example, if $n=1$, $\e_1\neq0$, and $lim_{c\to \infty}\frac{a(c)}{c^4}=0$, then the Cauchy problem is globally
well-posed in $H^s$, for  $s > \frac{3}{2}$ \cite{ka83}.
On the other hand, if for instance $a(u) =u^4$, some solutions do blow-up at a finite time \cite{martel2004}.\\

Since the general pattern is unknown, here we analyze the KdV Cauchy problem
\begin{equation}\label{eq:kdv}
 u_t=uu_x+\e u_{xxx}, \quad u|_{t=0}=\varphi \in H^s, \quad s\geq 3.
\end{equation}
If $\eps =0$, the solution of above Cauchy problem
develops a gradient catastrophe singularity at a finite time $t=t_c>0$.
Here $t_c$ coincides with the supremum of the positive time $t$ for which the solution of
the Cauchy problem can be continued.
Conversely, if $\e \neq 0$ the Cauchy problem is globally well-posed \cite{ka83}. \\

We now show that the solution of the Cauchy problem of KdV is
continuous with respect to $\e$ in any time interval $[0,T]$, with $T$ strictly smaller than the critical time $t_{c}$.
This is a simple
corollary of Theorem \ref{thm:localgenkdv}:

\begin{Thm}\label{thm:kdvcontinuity}
Let $T$ be any positive time smaller than the critical time $T<t_c$ and
let $u(t;\e) \in C([0,T],H^s) \cap C^1([0,T], H^{s-3})$ be the unique solution of the KdV Cauchy problem \eqref{eq:kdv}.
Then $\e \to u(t;\e)$ is a continuous map from $\bb{R}$ to $C([0,T],H^s) \cap C^1([0,T], H^{s-3})$.

\begin{proof}
The proof follows from Theorem \ref{thm:localgenkdv} and a standard continuation argument.
\end{proof}
\end{Thm}

\begin{Cor}\label{cor:continuityhopf}
Let $s\geq 3$, $T$ be any positive time smaller than the critical time $T<t_c$ and
let $u(t;\e)$ be the unique solution of the KdV Cauchy problem \eqref{eq:kdv}.
Then 
\begin{equation}\label{eq:continuityhopf}
 lim_{\e \to 0}\avs{u(t;\e) -v^0(t)}{s}, \mbox{ uniformly in } t \in [0,T].
\end{equation}
Here, $v^0$ is the unique solution of the Cauchy problem for the Hopf equation
$$v^{0}_{t}=v^{0}v^{0}_{x},\qquad  v^{0}|_{t=0}=\varphi.$$
\end{Cor}

\section{The $\e$-expansion of KdV solutions}\label{sec:kdvderivatives}
This section is devoted to study the differentiability of solutions of equation \eqref{eq:genkdv} with respect to $\e$.
For simplicity, we consider in detail the case of KdV only; as explained in Theorem \ref{thm:genmain} below, there is no trouble in extending the results to the whole class \eqref{eq:genkdv}.\\

We show that if $s \geq 3N+3$, then the solution of the Cauchy problem of KdV (\ref{eq:kdv}) is
$N$-times differentiable with respect to $\e$ in any time interval $[0,T]$, with $T$ strictly smaller than the critical time $t_{c}$. Before our main theorem, 
state a technical lemma.

\begin{Lem}\label{lem:A}
Let  $f:  \bb{R} \to C([0,T],H^{\frac{1}{2}})$ and $g:\bb{R} \to C([0,T],H^{\frac{3}{2}})$ be two functions,
and take $X=L^2$, $Y=H^s$, with $s\geq 3$. Then, the family of linear operators
\begin{equation}\label{eq:A}
 A(\e_1,\e_2,\e_3)=f(\e_1)\partial_x+g(\e_2)+\e_3\partial^3_x
\end{equation}
satisfy the conditions (i) through (iv) of Theorem \ref{thm:linear}.
Moreover, if $(\e_1^n,\e_2^n,\e_3^n) \to (\e_1,\e_2,\e_3)$ is any converging sequence, then the sequence of operators $A^n=A(\e_1^n,\e_2^n,\e_3^n)$ satisfy conditions (vi) and (vii) of Theorem \ref{thm:linearpert}.
\end{Lem}
\begin{proof}
The operator $A(\e_1,\e_2,\e_3)$ was considered in Examples \ref{ex:fd} and \ref{ex:multiplication} above, where it was shown that
it belongs to $G(X,1,\beta+\beta')$ with $\beta=\frac{1}{2}\sup_{x\in \bb{R}}\norm{f_x}$ and
$\beta'=\sup_{x\in \bb{R}}\norm{g}$.
The verification of conditions (i) through (vii) follows the same steps as the proof of Theorem \ref{thm:localgenkdv}.
\end{proof}

\begin{Thm}\label{thm:main}
Let $N=\lfloor s/3-1\rfloor$ and $[0,T]$ be a time interval such that $0<T<t_c$.
Let 
$$u: \bb{R} \to C([0,T],H^s) \cap C^1([0,T],H^{s-3})$$
 be
the map that associates to $\e \in \bb{R}$ the unique solution of the KdV Cauchy problem (\ref{eq:kdv}).
Then, there exist and are continuous the maps
$$u^{(k)}: \bb{R} \rightarrow C([0,T],
H^{s-3k})\cap C^1 ([0,T],H^{s-3(k+1)}),$$ 
for $k=1,\dots,N,$ defined as 
$$u^{(k)}(\e)=\frac{d^{k}u(\e)}{d\e^{k}}.$$ 
Fixed $\e\in \bb{R}$ and $k\in\bb{N}$, $1\leq k \leq N$,  then the function $u^{(k)}(\e)$ satisfies the following Cauchy problem
\begin{align}
&u^{(k)}_{t}=A_{\e}u^{(k)}+\sum_{j=1}^{k-1}\binom{k}{j}u^{(j)}u^{(k-j)}_{x}+k\,u^{(k-1)}_{xxx},\label{ukeqs}\\
&u^{(k)}|_{t=0}(\e)=0 ,
\end{align}
where the linear operator $A_{\e}$ is defined as
\begin{equation}\label{eq:ae}
A_{\e}=u(\e)\partial_{x}+u_{x}(\e)+\e \partial_{x}^{3},
\end{equation}
and we use the convention $u^{(0)}(\e) =u(\e)$.
\end{Thm}
\noindent
Note that \eqref{ukeqs} is a linear non-homogeneous differential equation.

\begin{proof}
If $N=0$, then the first part of the theorem follows from Theorem \ref{thm:kdvcontinuity}, while the second part is empty.
Assume $N\geq 1$, we now prove differentiability. To this aim we introduce the difference quotient
$$
u^{(1)}(\e,h)=\frac{u(\e+h)-u(\e)}{h} \; ,
$$
and a simple computation shows that $u^{(1)}(\e,h)$ satisfies the equation
$$
u^{(1)}_t(\e,h)=A(\e,h)u^{(1)}(\e,h)+u_{xxx}(\e+h) ,\qquad  u^{(1)}(\e,h)|_{t=0}=0 \; ,
$$
where $A(\e,h)= u(\e)\partial_{x}+u_{x}(\e+h)+\e \,\partial_{x}^{3}$. In the limit $h \to 0$, the above equation converges to (\ref{ukeqs}), with $k=1$.
This is a linear non-homogeneous equation, with forcing term $u_{xxx}(\e)$. 
From Theorem \ref{thm:kdvcontinuity}, we have that $u_{xxx}(\e)$ is a continuous function from $\bb{R}$ to $C([0,T],H^{s-3})$.

Hence, we can prove the convergence of $\lim_{h\to 0}u^1(\e,h)$ using the perturbation Theorem  \ref{thm:linearpert},
provided that:  i) we look for solutions of (\ref{ukeqs}) lying in the same space of the forcing term, namely $H^{s-3}$, and
ii) the forcing term belongs to
$D(A_{\e})$. The latter condition holds since $s\geq 6$ by hypothesis.

Due to Theorems \ref{thm:linear} and \ref{thm:linearpert} and to Lemma \ref{lem:A}, we conclude that
$$u^{(1)}(\e):= \lim_{h\to0}u^{(1)}(\e,h)$$ 
solves (\ref{ukeqs}) and maps $\bb{R}$ continuously to $C([0,T],H^{s-3})\cap C^1([0,T],L^2)$. Moreover, we have that
the function $u^{(1)}$ maps $\bb{R}$ continuously into $$C([0,T],H^{s-3}) \cap C^1([0,T],H^{s-6}).$$
Indeed, $u^{(1)}_t(\e)$ equals $A_{\e}u^{(1)}+u_{xxx}(\e)$, and the operator $A_\e$
maps any continuous function $\bb{R} \to C([0,T],H^{s-3})$ to a continuous function $\bb{R} \to C([0,T],H^{s-6})$. The last statement can be proved in a similar way as in the proof of Theorem \ref{thm:localgenkdv}.  \\

We continue the proof by induction on the order of the derivative. Suppose the thesis is valid for $i=1,\dots,k <N$.
As before, we define the difference quotient
$$
u^{(k+1)}(\e,h)=\frac{u^{(k)}(\e+h)-u^{(k)}(\e)}{h} \; ,
$$
that satisfies the non-homogeneous linear equation

{\setlength\arraycolsep{1pt}
\begin{eqnarray}
&u^{(k+1)}_t(\e,h)=&A(\e,h)u^{(k+1)}(\e,h)  + u^{(k)}_{xxx}(\e+h)+f^{k+1}(\e,h) \notag\\
&{}&+u^{(1)}(\e)u^{(k)}_{x}(\e)+ u^{(1)}_x(\e)u^{(k)}(\e),  \label{uk+1eq}\\
&u^{(k+1)}(\e,h)|_{t=0}&=0, {}\notag
\end{eqnarray}
}
where
\begin{align}
 f^{k+1}(\e,h)= \frac{1}{h}&\left( \sum_{j=1}^{k-1}\binom{k}{j}u^{(j)}(\e+h)u^{(k-j)}_{x}(\e+h)+k\,u^{(k-1)}_{xxx}(\e+h)-\right. \notag \\
&\left.\sum_{j=1}^{k-1}\binom{k}{j}u^{(j)}(\e)u^{(k-j)}_{x}(\e)+k\,u^{(k-1)}_{xxx}(\e)\right).
\end{align}
The non-homogeneous term of equation \eqref{uk+1eq} belongs to $C([0,T],H^{s-3(k+1)}),$ and it continuously depends on $\e$.
In the limit $h \to 0$, the quantity $f^{k+1}(\e,h)$ converges in $C([0,T],H^{s-{(3k+3)}})$, continuously with respect to $\e$,
to $$f_{\e}^{k+1}(\e) := \frac{d}{d\e}\left(\sum_{j=1}^{k-1}\binom{k}{j}u^{(j)}(\e)u^{(k-j)}_{x}(\e)+k\,u^{(k-1)}_{xxx}(\e)\right).$$
Hence, the same reasoning as in the case of $u^{(1)}(\e)$ shows that 
$$u^{(k+1)}(\e):= \lim_{h\to0}u^{(k)}(\e,h)$$
solves the equation
{\setlength\arraycolsep{1pt}
\begin{eqnarray}
& u^{(k+1)}_t(\e)=&A(\e)u^{(k+1)}(\e)+u^{(k)}_{xxx}(\e)+f_{\e}^{k+1}(\e) \notag \\ 
& &  +u^{(1)}(\e)u^{(k)}_{x}(\e)+ u^{(1)}_x(\e)u^{(k)}(\e), \label{eq:k+1proof}\\
&u^{(k+1)}(\e)|_{t=0}&=0, \notag
\end{eqnarray}
}
and it maps $\bb{R}$ continuously to $C([0,T], H^{s-3(k+1)})\cap C^1 ([0,T],H^{s-3(k+2)})$.
It is a simple computation to show that (\ref{eq:k+1proof}) coincides with (\ref{ukeqs}).
\end{proof}

Theorem \ref{thm:main} shows that if the initial datum of the KdV
equation lies in the Sobolev space
$H^s$, then the solution of the Cauchy problem is $N-$times
differentiable with respect to $\e$, for $N=\lfloor s/3-1\rfloor$.
Consequently, if the initial datum lies in all the Sobolev space -- e.g. it
belongs to the Schwartz class -- then the solution of the Cauchy problem
is smooth with respect to $\e$. In particular, the solution admits an
asymptotic expansion in power series of $\e$. More precisely, we have the following corollary of Theorem \ref{thm:main}.

\begin{Cor}
\label{cor:hinfty}
Let $\varphi \in H^{\infty}=\cap_{s\geq0}H^s$, $T>0$ be any positive
time smaller than
the critical time $T<t_c$ and $u: \bb{R} \to C([0,T],H^{\infty})$ be
the solution of the Cauchy problem (\ref{eq:kdv}).
Then
\begin{itemize}
\item[(i)]$u: \bb{R} \to C([0,T],H^{\infty})$ is a smooth function (of $\e$).
\item[(ii)]In $\e=0$, $u$ admits an asymptotic expansion in power
series of $\e$:
\begin{equation}\label{formula:asymptotic}
u(\e) \sim \sum_{k=0}^{\infty} v^{k}\,\e^k.
\end{equation}
where $v^{0}=u(0)$ is the solution of the Cauchy problem
\begin{align}
&v^{0}_{t}=v^{0}v^{0}_{x}\label{eq:hopf}\\
&v^{0}|_{t=0}=\varphi,
\end{align}
for the Hopf equation, and $v^{k}=u^{(k)}(0)$ is the solution of the $k$-th linear  Cauchy problem
(\ref{ukeqs}) when $\e=0$, that is:
\begin{align}
&v^{k}_{t}=\sum_{j=0}^{k}\binom{k}{j}v^{j}v^{k-j}_{x}+k\,v^{k-1}_{xxx},\label{ukeqshopf}\\
&v^{k}|_{t=0}=0,\qquad k\geq 1.
\end{align}
\end{itemize}
 
\end{Cor}
Note that a similar Theorem was proven for the defocusing Nonlinear Schr\"odinger equation \cite{grenier98}.\\

Below, we state the analogue of Theorem \ref{thm:main} for the general equation \eqref{eq:genkdv}
and we give a sketch the proof. The details of the full proof, which is rather long, will be given elsewhere.

\begin{Thm}\label{thm:genmain}
Let $U$ be an open subset of $\bb{R}^n$, $[0,T]$ be a time interval
such that the Cauchy problem is locally well-posed and continuous with respect to $\e \in U$ and
let
$$u: \bb{R} \to C([0,T],H^s) \cap C^1([0,T],H^{s-(2n+1)})$$
be the map that associates to $\e \in U$ the unique solution of the Cauchy problem for (\ref{eq:genkdv}).
Moreover, let $K=\sum_{i=1}^{n}N_i(2i+1)$ and $N=\sum_{i=1}^{n}N_i$. If
$s-K \geq 2n+1$, then the partial derivative
$$\frac{\partial^N u(\e)}{\partial\e_1^{N_1}\dots\partial\e_n^{N_n}} $$ exists
 in $C([0,T],
H^{s-K})\cap C^1 ([0,T],H^{s-K-(2n+1)})$
and is continuous with respect to  $\e \in U$.
\end{Thm}
\begin{proof}
The Theorem can be proven along the very same lines of the proof of the analogue Theorem \ref{thm:main} for KdV. 
More precisely, it is possible to prove existence and continuity of the partial derivative, by showing that
it satisfies a linear non-homogeneous equation. Note that the condition $s-K \geq 2n+1$ implies that the forcing term belongs to the domain of $\partial_x^{2n+1}$.
\end{proof}

\section{Hamiltonian perturbation of quasilinear conservation laws}\label{sec:dubrovin}
We now consider the results of the previous sections in the setting of the general construction,
proposed by Dubrovin and Zhang \cite{du06,du10,duzh01}, of Hamiltonian regularization of the quasilinear conservation law:
\begin{equation}\label{eq:genhopf}
u_{t}=a(u)\,u_{x}, \qquad u|_{t=0}=\varphi,
\end{equation}
where $a$ and the initial value $\varphi$ are assumed to be smooth functions, and $\varphi$ is either periodic
or rapidly decreasing at infinity. We discuss the aspects of the Dubrovin-Zhang construction which are more related with the present paper;
in the next section we will show how the results obtained in Section \ref{sec:continuity} and \ref{sec:kdvderivatives} provide a
rigorous justification to this method for a particular class of equations of type \eqref{eq:genkdv}.

Let us consider equation \eqref{eq:genhopf}. It is well known that this equation can formally be written as a
Hamiltonian system
\begin{equation}
u_{t}=\left\{u,H\right\}=\partial_{x}\frac{\delta H}{\delta u(x)},
\end{equation}
with Hamiltonian
$$H=\int_{\bb{R}}h(u(x))\,dx,\qquad h''(u)=a(u),$$
and where the Euler-Lagrange operator is defined as
$$\frac{\delta H}{\delta u(x)}=\sum_{k\geq 0}(-1)^{k}\frac{d^{k}}{d x^{k}}\frac{\partial \,h}{\partial u_{x}^{(k)}},$$
for any local functional $H=\int h(u,u_{x},u_{xx},\dots)\,dx$. The above Poisson bracket is given by
$$\left\{H_{1},H_{2}\right\}=\int_{\bb{R}}\frac{\delta H_{1}}{\delta u(x)}\partial_{x}\frac{\delta H_{2}}{\delta u(x)}\,dx,$$
for any pair of functionals $H_{i}=\int h_{i}(u, u_{x},u_{xx},\dots)\,dx$, $i=1,2$. \\

Following Dubrovin, by \emph{Hamiltonian regularization} (or \emph{perturbation}) of the quasilinear conservation law
\eqref{eq:genhopf} we mean an expression
\begin{equation}\label{eq:genperturbation}
u_{t}=\left\{u,\tilde{H}\right\}=\partial_{x}\frac{\delta \tilde{H}}{\delta u(x)}, \qquad u|_{t=0}=\varphi,
\end{equation}
where the Hamiltonian is given by a formal series
$$\tilde{H}=H+\sum_{k\geq 1}H_{k}\,\e^{k},\qquad H_{k}=\int h_{k}(u;u_{x},\dots,u^{(k)}_{x})\,dx,\quad k\geq 1,$$
for some $h_{k}$, which are assumed to be differential polynomials in the derivatives.
In addition, the solutions of equation \eqref{eq:genperturbation} are sought  to be of the form
\begin{equation}\label{expu}
u(x,t)=\sum_{i\geq 0} v^{i}(x,t)\,\e^{\,2i},
\end{equation}
with coefficients  $v^{i}$ smooth functions of $x$ and $t$. Within this setting, all identities are understood in the sense of
\emph{formal power series} in $\e$ - they are assumed to hold identically at every order in $\e$. Therefore, the perturbed Hamiltonian and
the solutions are not required to be convergent (neither asymptotic) series. Note, however,
that the initial values of the perturbed and the unperturbed equations are assumed to be the same. In particular, the function $\varphi$ is independent of $\e$, and from the expansion \eqref{expu} we deduce the identities
\begin{equation}\label{eq:ic}
v^0(x,0)=\varphi(x),\, \qquad v^i(x,0)=0, \, \quad \forall \,i \geq 1.
\end{equation}

A primary task in the Dubrovin-Zhang approach is the classification of Hamiltonian perturbations, which is performed
by considering equations \eqref{eq:genperturbation} modulo \emph{quasi-Miura transformation}. These are transformations
of the form:
\begin{equation}\label{eq:miura}
u\longmapsto v=\sum_{k\geq 0}\e^{k}F_{k}(u;u_{x},\dots,u^{(k)}_{x}),
\end{equation}
where the functions $F_{k}$ are rational with respect to the derivatives. A partial result is given by the following
\begin{Thm}\cite{du06}\label{thm:borisclass}
Any Hamiltonian perturbation of equation \eqref{eq:genhopf} of order $\e^{4}$ can be reduced by a Miura-type
transformation to an equation of the form \eqref{eq:genperturbation}, with Hamiltonian
\begin{gather}
\tilde{H}=\int \tilde{h}(u;u_{x},u_{xx}, \e) dx, \qquad h''=a, \label{eq:normham}\\
\tilde{h}=h-\frac{\e^2}{2}\,c\,h'''\,u_{x}^2+
\e^4\left[\left(p\,h'''+\frac{3}{10}\,c^{2}\,h^{(4)}\right) u_{xx}^2\right. \notag\\
\left.- \left(\frac{c\,c''}{8}\,h^{(4)}+\frac{c\,c'}{8}\,h^{(5)}+\frac{c^{2}}{24}h^{(6)}+
\frac{p'}{6}h^{(4)}+\frac{p}{6}h^{(5)}-s\,h'''\right)u_{x}^4 \right],\notag
\end{gather}
for arbitrary functions $c(u)$, $p(u)$, $s(u)$.
\end{Thm}
Let us now consider in more detail the solutions of the perturbed equation \eqref{eq:genperturbation}, which -- after Theorem \ref{thm:borisclass} -- we will consider together with a  Hamiltonian of the form \eqref{eq:normham}.  By expanding both sides of equation \eqref{eq:genperturbation} according
to the Ansatz \eqref{expu}, in first approximation one obtains
\begin{equation}\label{eq:vhopf} 
v^{0}_{t}=a(v^{0})\,v^{0}_{x},
\end{equation}
which says that $v^{0}$ must be a solution of the unperturbed equation \eqref{eq:genhopf}.  Accordingly, from the higher order coefficients one obtains an infinite set of \emph{linear non-homogeneous}
equations (or \emph{transport equations}) for the coefficients $v^{k}(x,t)$, which can be solved recursively
starting from the solution of \eqref{eq:vhopf}. For instance, the equation for $v^{1}$ turns out to be
\begin{equation}\label{eq:tr1}
v^{1}_{t}=\partial_{x}\left(a \,v^{1}+c\, a'\, v^{0}_{xx}+\frac{1}{2}\left(c\,  a''+c'\,  a'\right)\, \left(v^{0}_{x}\right) ^{2}\right), \quad v^1|_{t=0}=0
\end{equation}
where
$a=a\left(v^{0}\right)$, $c=c\left(v^{0}\right)$.
\begin{Rem}
In the class of initial data considered in the present paper, any solution of equation \eqref{eq:vhopf} develops a singularity at a finite time $t=t_c>0$, known as time of gradient catastrophe. An important aspect of Dubrovin's theory is concerned with the study  of the solution of \eqref{eq:genperturbation} in a neighborhood of the critical time $t_c$, in order to show how the singularity is regularized by the dispersive perturbations \cite{du06}.

This, however, is out of the scope of our present method of investigation.
In what follows we study the series \eqref{expu} for any time interval
$[0,T]$ strictly smaller than the critical time: $T<t_c$.
\end{Rem}

\begin{Exa}
The KdV equation 
$$u_{t}=u\,u_{x}+\e^{2}\,u_{xxx}$$
is obtained from \eqref{eq:genperturbation}, \eqref{eq:normham} by choosing $h(u)=\tfrac{1}{6}u^{3},$ $c(u)=1$, and $p(u)=s(u)=0$. In this case, a simple computation shows that the Cauchy problems for the  transport equations are given by
\begin{align*}
&v^{k}_{t}=\sum_{j=0}^{k}\binom{k}{j}v^{j}v^{k-j}_{x}+k\,v^{k-1}_{xxx},\qquad k\geq 0,\\
&v^{0}|_{t=0}=\varphi,\quad v^{k}|_{t=0}=0.
\end{align*}
One can -- in principle -- solve these equations recursively.
Note that these Cauchy problems coincide with \eqref{eq:hopf}, \eqref{ukeqshopf}. 
\end{Exa}
From the above discussion it follows that -- at least in principle -- all coefficients of the expansion \eqref{expu} can be obtained once a solution $v^{0}$ of equation \eqref{eq:vhopf} is known.  The method followed in \cite{du06,du10} (see also the older result \cite{bagaib89}) to find solutions of the perturbed equation \eqref{eq:genperturbation} is to construct a quasi-Miura transformation, relating the solution $v^{0}(x,t)$ of the unperturbed equation \eqref{eq:vhopf} to the the solution $u(x,t)$ of perturbed equation \eqref{eq:genperturbation}. The required transformation has been suggested in \cite{du06} to be of the form:
\begin{equation}\label{cantrans}
v^{0}\longmapsto u=v^{0}-\e\,\left\{v^{0}(x),K\right\}+ \e^2\left\{\left\{v^{0}(x),K\right\},K\right\}+\dots
\end{equation}
where the functional $K$, up to order $4$ in $\e$, is given by
\begin{equation}
K=-\!\!\int\left[\e\frac{c(v^{0})}{2}\,v^{0}_{x}\log{v^{0}_{x}}+\e^{3}\left(\frac{c(v^{0})^{2}}{40}
\left(\frac{v^{0}_{xx}}{v^{0}_{x}}\right)^{3}\!\!-\frac{p(v^{0})}{4}\frac{\left(v^{0}_{xx}\right)^{2}}{v^{0}_{x}}\right)\right] dx.
\end{equation}
This in particular implies
\begin{equation}\label{eq:borisu1}
v^{1}=\frac{1}{2}\partial_{x}\left(c(v^{0})\frac{v^{0}_{xx}}{v^{0}_{x}}+c'(v^{0})\,v^{0}_{x}\right),
\end{equation}
and a direct substitution shows that the above function satisfies equation \eqref{eq:tr1}, provided $v^{0}$ satisfies \eqref{eq:vhopf}.
Note, however, that the function \eqref{eq:borisu1} is bounded only for monotone solutions of equation \eqref{eq:vhopf}. Furthermore, \eqref{eq:borisu1} does not satisfy the required initial condition \eqref{eq:ic}.

\section{Solutions of the transport equations}\label{sec:solutionv1}

The classification problem of Hamiltonian perturbations, together with
the quasi-Miura transformation discussed above, are main ingredients
of the Dubrovin--Zhang constructions before the critical time, which is the
time-span we are interest in the present paper. As already noticed,
this approach is mainly based on identities of formal power series,
and indeed, if for a given Hamiltonian the corresponding equation
\eqref{eq:genperturbation} is just a formal series,
then the only way to construct a (formal) solution seems to be through
the use of a formal series, for instance like \eqref{expu}.

However, if the Hamiltonian perturbation \eqref{eq:genperturbation} is
well-defined -- for example if the Hamiltonian $\tilde{H}$
is given by \eqref{eq:normham} -- then the resulting equation may
happen to be locally well-posed in some function space, say $H^s$,
for $s$ big enough.
At the same time, the solution may be $N-$differentiable with respect to
$\e$ and the formal series \eqref{eq:genperturbation} may just be
-- up to order $N$ -- the Taylor expansion of the actual solution.\\

Using the results of Section \ref{sec:continuity} and
\ref{sec:kdvderivatives} we can show that this is true
if the Hamiltonian perturbation \eqref{eq:genperturbation} coincides with a
generalised KdV equation \eqref{eq:genkdv}.
This is the case for the general equation (\ref{eq:genkdv}), provided $n\leq 2$.
Indeed, the following Lemma holds:\\

\begin{Lem}\label{lem:borisvsus}
The equation
\begin{equation}\label{eq:borisgenkdv}
u_{t}=a(u)\,u_{x}+\e^{2}\,\alpha\,u_{xxx}+\e^{4}\,\beta\,u_{xxxxx},
\end{equation}
which is obtained by \eqref{eq:genkdv} in the case $n=2$ and
$\e_{1}=\alpha\,\e^{2}, \e_2=\beta\,\e^4$, coincides with
the Hamiltonian perturbation (\ref{eq:genperturbation})
with Hamiltonian \eqref{eq:normham}, provided the coefficients are
chosen in the following way:
\begin{gather}\label{eq:parameters}
c=\frac{\alpha}{a'},\quad
p=\frac{\beta}{2\,a'}-\frac{3\,\alpha^{2}}{10}\frac{a''}{(a')^{3}},\\
\nonumber
s=\alpha^{2}\,\left(\frac{2}{5}\,\frac{(a'')^{3}}{(a')^{5}} -\frac
{7}{20}\frac{a''\,a'''}{(a')^{4}}
+\frac{1}{24}\frac{a''''}{(a')^{3}}\right)
-\frac{\beta}{12}\left(\frac{(a'')^{2}}{(a')^{3}}-\frac{a'''}{(a')^{2}}\right).
\end{gather}
\end{Lem}

\begin{Thm}\label{thm:expansion}
Let $\varphi \in H^{s}, s\geq 11 $, $U$ a neighborhood of $\e=0$,
$u(\e)$ be the solution of the Cauchy problem
\eqref{eq:borisgenkdv} with initial data $\varphi$ and $[0,T]$ be a
time interval such that the Cauchy problem is locally well-posed for
any $\e \in U$ and continuous with respect to $\e \in U$.
Moreover, let $s\geq 5+6N$. Then $u(\e)$ has a Taylor expansion in $\e=0$, up to order $4N$, of the form
\begin{equation}\label{formula:genasymptotic}
u(\e) = \sum_{k=0}^{2N} v^{k}\,\e^{2k}+ r(\e),
\end{equation}
where
\begin{equation*}
v^k \in C([0,T],H^{s-3k}),\qquad r(\e) \in C([0,T], H^{s-6N}).
\end{equation*}
Here $v^{0}=u(0)$ is the solution of the Cauchy problem
for unperturbed equation \eqref{eq:vhopf}, while
$v^{1}$ satisfies the first transport equation \eqref{eq:tr1}
with parameters \eqref{eq:parameters},
$v^k, k>1$ is the solution of the higher transport equations.
The remainder term $r$ is $o(\e^{4N})$.

Moreover, if $\alpha=0$ then the same is true, under weaker hypothesis. Namely, if $s \geq 5+5N$, then $v^k=0$ for
$k$ odd, $v^{2l} \in C([0,T],H^{s-5l})$,
and $r:\bb{R} \to C([0,T, H^{s-5N}])$ is continuous, with $r=o(\e^{4N})$.
\begin{proof}
The proof follows from Theorem \ref{thm:genmain}.
\end{proof}
\end{Thm}

\begin{Rem}
Due to Theorem \ref{thm:localgenkdv}, $T_Q >0$ for any $Q\subset \bb{R}$.
If, for  $(\e,\alpha,\beta) \neq (0,0,0)$ the Cauchy problem is
globally well-posed \cite{kenig91}, then the time $T_Q$  in the above
theorem can be chosen to be any positive
time smaller than the critical time $t_{c}$ of the unperturbed
equation (\ref{eq:vhopf}). This is the case, for instance, if
$a(u)=u$. 
\end{Rem}

In the following theorem we provide a correction to the formula \eqref{eq:borisu1} for the first coefficient $v^{1}$. Our formula turns out to be valid for solutions of \eqref{eq:vhopf} which are not monotone, and satisfies the correct initial value. 

\begin{Thm} Consider the Cauchy problem for equation \eqref{eq:vhopf} with initial datum $v^{0}(0)=\varphi\in H^{s}$, $s\geq 3$, denote by $t_{c}$ be the associated critical time, and let $v^{0}(x,t)$, with $(x,t)\in\bb{R}\times [0,t_{c}),$ be its unique classical solution. Then, the solution of the linear transport equation \eqref{eq:tr1}, with $v^{0}$ given above and initial datum $v^{1}(0)=0$ is
$$v^{1}(x,t)=\frac{\partial}{\partial x}\left(\frac{\delta \tilde{K}_{t}[u]}{\delta u(x)}_{|u=v^{0}(x,t)}\right),\qquad (x,t)\in\bb{R}\times [0,t_{c}),$$
where the family of functionals $\tilde{K}_{t}$, with $t\in\bb{R}$, is defined by
$$\tilde{K}_{t}[u]:=-\frac{1}{2}\int_{\bb{R}} c\,(u)\,u_{x}\,\log{\Big(1+t\,a'\left(u\right)u_{x}\Big)}\,dx,$$
for every $u\in H^{s}$, with $\avs{u}{s}$ small enough. 
\end{Thm}
\begin{proof}
The explicit form of the function $v^{1}$ stated in the theorem is given by
\begin{equation}\label{eq:v1}
v^{1}=\frac{t}{2}\,\frac{\partial}{\partial x}\! \left(\frac{\left(c\,a'\right)'(v^{0}_{x})^{2}+2\,c\,a'\,v^{0}_{xx}+t\,c\,(a')^{2}v^{0}_{x}\,v^{0}_{xx}+t\,c'\,(a')^{2}(v^{0}_{x})^{3}}{(1+t\,a'\,v^{0}_{x})^{2}}\right),
\end{equation}
where the functions $a$, $c$, and the corresponding derivatives are evaluated at $u=v^{0}(x,t)$. A direct calculation shows that this function satisfies equation \eqref{eq:tr1} with the correct initial value.
\end{proof}

\begin{Rem}
Note that the above theorem holds for any choice of the functions $a$ and $c$. This fact suggests that similar results of the one obtained in Theorem \ref{thm:expansion}  remain true for a generic Hamiltonian perturbation of the quasilinear conservation law \eqref{eq:genhopf}.
\end{Rem}

\begin{Rem}
Formula \eqref{eq:v1} has been obtained making use of the so called \lq string equation\rq , introduced in the setting of Hamiltonian perturbations of nonlinear PDEs by Dubrovin  \cite{du06}. Although heuristic, the use of the string equation turns out to be a very powerful method for describing solutions of the perturbations both before the critical time and in a neighborhood of it.
\end{Rem}

\bibliographystyle{plain}
\bibliography{biblio}

\end{document}